\documentclass[a4paper,11pt,twoside]{article}
\usepackage{times}
\usepackage[ansinew]{inputenc}
\usepackage[T1]{fontenc}

\usepackage{geometry}
\usepackage{fmtcount}

\addtocounter{footnote}{-1}

\usepackage{amsmath, amsthm, amsfonts}

\newtheorem{thm}{Theorem}[section]

\newtheorem{lem}[thm]{Lemma}
\newtheorem{prop}[thm]{Proposition}

\newtheorem{exa}[thm]{Example}

\newcommand{\co}{\mathrm{conv}}
\newcommand{\N}{\mathbb{N}}
\newcommand{\R}{\mathbb{R}}

\newcommand{\dia}{\mathrm{diag}}

\title{Some remarks on $p$-median problem}

\author{ Tijani Pakhrou}

\date{}

\date{\small Department of Mathematics\\
  \small Faculty of Sciences, \small University of Alicante\\
  \small Alicante, \small Spain\\
  tijani.pakhrou@ua.es}

\begin{document}

\maketitle

\section*{Abstract}

\

It is shown that for a given Banach space $X$ and a subspace $Y$ weakly
$\mathcal{K}$-analytic, $L_p(\mu,Y)$ is $p$-simultaneously
proximinal in $L_p(\mu,X)$ whenever $Y$ is $p$-simultaneously
proximinal in $X$.

\begin{itemize}
\item[]{\bf Mathematics Subject Classification (2020)}. 41A28, 41.40.
\item[] {\bf Keywords:} Simultaneous approximation, Best approximation.

\end{itemize}

\section{Introduction}

Throughout this paper, $(X,\|\cdot\|)$  is a real Banach space,
   $(\Omega, \Sigma, \mu)$  a complete probability space
    and  we assume $1\leq p<\infty$. We denote by $L_p(\mu,X)$
   the Banach space of all Bochner $p$-integrable
   functions defined on $\Omega$ with values in $X$ endowed with the
   usual $p$-norm  (Diestel and Uhl, 1977).

Let $Y$ be a subset of $X$ and $n$ a positive integer.
Let $A=\{a_1,\ldots,a_n\}\subset X$ and $m\in[1,\infty)$, we say that
$y_0\in Y$ is a {\em relative $m$-median of $A$ in $Y$} or {\em relative $m$-center of $A$ in $Y$}
 if
\begin{align*}
      \sum_{i=1}^n\|a_i-y_0\|^m \leq
      \sum_{i=1}^n\|a_i-y\|^m, \ \ \forall \, y\in Y.
\end{align*}
We denote by $Z_Y^m(A)$ the set of relative $m$-{\em centers}
of $A$ in $Y$. When $Y=X$, then $Z^m(A)$ denotes the
set of $m$-centers of $A$.
Note that when $A$ has one or two points, the $m$-centers of $A$ are trivial. Indeed, if we take only one point, obviously its $m$-center itself. If we take a couple
of points,  it is easy to prove that their $m$-centers are always in their
convex hull, i.e. the segment defined by these two points.
Moreover, if $X$ has dimension not exceeding $2$, by Helly theorem
(Mendoza and Pakhrou, 2003) we have
$$
Z^m(\{a_1,\ldots,a_n\})\cap\co(\{a_1,\ldots,a_n\})\neq\emptyset,
$$
for every  $a_1,\ldots,a_n\in X$. However, if we consider a set of three points in a
space of dimension greater than two, we lose information about the existence and
location of $m$-centers.

If every $n$-tuple of vectors $a_1,\ldots,a_n\in X$
admits a relative $m$-center in $Y$, then $Y$ is
said to be {\em $m$-simultaneously proximinal in $X$}.
Of course, for $n=1$ the notion of relative $m$-center is just that of {\em best
approximation} and the concept of $m$-simultaneously proximinal coincides with the {\em proximinality} (Mendoza, 1998).

One of the questions studied about the relative $m$-center
is the type of subspaces that are $m$-simultaneously proximinal
in a Banach space.
Natural subspaces of $L_p(\mu,X)$ are either of the kind
$L_p(\mu,Y)$, where $Y$ is a closed subspace of $X$, or of the
kind $L_p(\mu_0,X)$, where $\mu_0$ is the restriction of
$\mu$ to a certain sub-$\sigma$-algebra $\Sigma_0$ of
$\Sigma$.
The main problem we present in this paper is:
{\em Let $1\leq p<+\infty$ and $1\leq m<+\infty$. If \, $Y$ is a subspace of $X$,  $m$-simultaneously proximinal in $X$ and weakly
$\mathcal{K}$-analytic, is
$L_p(\mu,Y)$ $m$-simultaneously proximinal in $L_p(\mu,X)$}?

We denote by  $\ell_m^n(X)$
the Banach space consisting of all $n$-tuple
  $$(x_1,\ldots,x_n)\in X^n=\overbrace{X\times X\times\cdots\times X}^{n \text{ times}}$$ with
  the norm
 $$
 \||(x_1,\ldots,x_n)\||_m:=\Big(\sum_{i=1}^n\|x_i\|^m\Big)^{\frac{1}{m}}, \ \ \
 m\in[1,+\infty).
 $$

Then it is obvious that
$y_0$ is a relative $m$-center of $A=\{a_1,\ldots,a_n\}\subset X$ in $Y$
if and only if $(y_0,\ldots,y_0)$ is a best approximation to $(a_1,\ldots,a_n)$ from the subspace
$$\dia\big(\ell_m^n(Y)\big):=\big\{(y,\ldots,y): y\in Y\big\}.$$

\section{Preliminaries}

Let $X$ be a nonempty set and $\mathcal{A}$ a class of
its subsets. We say that we are given a Souslin scheme
$\{A_{n_1,\ldots,n_k}\}$ with values in $\mathcal{A}$ if, to every
finite sequence $(n_1,\ldots,n_k)$ of natural numbers, there corresponds a set
$A_{n_1,\ldots,n_k}\in\mathcal{A}$. The Souslin operation over the class
$\mathcal{A}$ is a mapping such that to every Souslin scheme
$\{A_{n_1,\ldots,n_k}\}$ with values in $\mathcal{A}$, associates the set
$$
A=\bigcup_{(n_i)_{i\geq1}\in\N^{\N}}\bigcap_{k=1}^\infty A_{n_1,\ldots,n_k}.
$$
The sets $A\subset X$ of this form are called $\mathcal{A}$-Souslin.

We recall that a multivalued mapping $\Psi$
from a topological space
$X$ to the set of nonempty subsets of a topological space $Y$,
is called {\em upper semi-continuous} if for every $x\in X$ and every
open set $V$ in $Y$, containing the set $\Psi(x)$, there exists a
neighborhood $U$ of $x$ such that
$$
\Psi(U):=\bigcup_{u\in U}\Psi(u)\subset V.
$$

Let $X$ be a Hausdorff space and let $\N^{\N}$ be the space of sequences of positive integers endowed with the product topology. A subset $A$ of $X$ is called
$\mathcal{K}$-{\em analytic} if there exists an upper semi-continuous mapping
$\Psi$ on $\N^{\N}$ with values in the set of nonempty compact sets in $X$
such that the equality
$$
A=\bigcup_{(n_i)_{i\geq1}\in\N^{\N}}\Psi\big((n_i)_{i\geq1}\big)
$$
holds.

The metric projection on a subset $Y$ of a Banach space $(X,\|\cdot\|)$ is
the multivalued mapping given by
$$
P_Y(x)=\big\{y\in Y: \|x-y\|=\inf_{z\in Y}\|x-z\|\big\}
$$
for every $x\in X$. Note that $P_Y(x)$ may be the empty set for some $x\in X$.

Let $Y$, $Z$ be topological spaces. A selector for a multivalued
mapping $\phi: Y\longrightarrow 2^Z$ is a singlevalued mapping $f:Y\longrightarrow Z$
such that $f(y)\in \phi(y)$ for every $y\in Y$.
A set in a Hausdorff space is called {\em analytic} if it is the image
of a complete separable metric space under a continuous mapping.
We will say that a map $f:Y\longrightarrow Z$ is {\em analytic measurable} if the preimage of any
Borel subset of $Z$ belongs to the smallest $\sigma$-algebra containing
the analytic subsets of $Y$.

\begin{thm}[Cascales and Raja, 2003]\label{thm:2}
Let $X$ be a Banach space and $Y$ a proximinal and weakly
$\mathcal{K}$-analytic convex subset of $X$. Then, for every separable
closed subset $M\subset X$ the metric projection $P_Y|_M:M\longrightarrow 2^Y$
has an analytic measurable selector with separable range.
\end{thm}

Observe that the class of proximinal vector subspaces $Y$ which are $\mathcal{K}$-analytic
for the weak topology contains the reflexive subspaces,
proximinal separable subspaces, proximinal weakly compactly generated (WCG)
subspaces and proximinal quasi-reflexive subspaces, among others.

\begin{prop}\label{pro:1}
Let $X$ be a Banach space and $Y$ a weakly $\mathcal{K}$-analytic
subspace of $X$ which is $m$-simultaneously proximinal in $X$. Consider a
complete probability space $(\Omega, \Sigma, \mu)$ and
$f_1,\ldots,f_n:\Omega\longrightarrow X$ be $\mu$-measurable
functions. Then there exists a $\mu$-measurable function $g_0:\Omega\longrightarrow Y$
such that $g_0(s)$ is a relative $m$-center of
$\{f_1(s), \ldots,f_n(s)\}\subset X$ in $Y$,
for almost all $s\in\Omega$.
\end{prop}

\begin{proof}
Since $f_1,\ldots,f_n:\Omega\longrightarrow X$ are $\mu$-measurable
functions. Is immediate that the function
$$
\begin{array}{cccl}
f:    &  \Omega   &  \longrightarrow & \ell_m^n(X) \\
      &  s        & \longrightarrow  & f(s)=\big(f_1(s),\ldots,f_n(s)\big)
\end{array}
$$
is $\mu$-measurable. There exists $\Omega_0\in\Sigma$
such that $\mu(\Omega\backslash\Omega_0)=0$ and $f(\Omega_0)$ is a
separable subset of $\ell_m^n(X)$ for the norm topology.

On the other hand, countable product of $\mathcal{K}$-analytic sets is $\mathcal{K}$-analytic. Hence, since $Y$ is weakly $\mathcal{K}$-analytic, $\dia(\ell_m^n(Y))$ is weakly $\mathcal{K}$-analytic.

Denoting by $M$ the closure of $f(\Omega_0)$ in $\ell_m^n(X)$ for the norm topology,
from Theorem \ref{thm:2}, there exists
$$
h:M\longrightarrow \dia\big(\ell_m^n(Y)\big)
$$
an analytic selector with separable range for the norm topology,
of  the metric projection
$$
P_{{\dia(\ell_m^n(Y))}|_M}:M\longrightarrow 2^{\, \dia(\ell_m^n(Y))}.
$$
The function
$$
\begin{array}{cccl}
g:    &  \Omega   &  \longrightarrow & \dia\big(\ell_m^n(Y)\big) \\ \\
        &  s        & \longrightarrow  &
        g(s)=\left\{ \begin{array}{ccl}
       h\big(f(s)\big) & \text{ if } & s\in \Omega_0 \\
      0   & \text{ if } & s\in\Omega\backslash\Omega_0
\end{array}\right.
\end{array}
$$
is $\mu$-measurable, where the measurability follows from the fact that
$f|_{\Omega_0}$ is $\Sigma|_{\Omega_0}$-analytic measurable because
a complete probability space is stable by the Souslin operation,
(Kechris, 1995), and $h$ is analytic measurable.
By construction, $g$ satisfies,
$$
g(s)=
       h\big(f(s)\big)\in P_{{\dia(\ell_m^n(Y))}|_M}\big(f(s)\big),
        \ \ \text{ if } \ \ \ s\in \Omega_0.
$$
Then,
$$
\||f(s)-g(s)\||_m=\inf_{u\in \dia(\ell_m^n(Y))}\||f(s)-u\||_m,
\ \ \text{ if } \ \ \ s\in \Omega_0,
$$
 so $g(s)$ is a best approximation of $f(s)$ for almost all $s\in \Omega$.

Finally, we define the function $g_0:\Omega\longrightarrow Y$ by
$g_0(s)=y_s$ where $y_s$ is the component of $g(s)=(y_s,\ldots,y_s)\in\dia(\ell_m^n(Y))$ if $s\in \Omega_0$ and
$g_0(s)=0$ otherwise. It is clear that $g_0$ is $\mu$-measurable.

Thus, by taking $y\in Y$ we have
$$
\||f(s)-g(s)\||_m\leq\||f(s)-u\||_m
$$
for almost all $s\in\Omega$,
where $u=(y,\ldots,y)\in\dia(\ell_m^n(Y)) $. This implies that
$$
\sum_{i=1}^n\|f_i(s)-g_0(s)\|^m\leq \sum_{i=1}^n\|f_i(s)-y\|^m
$$
for almost all $s\in\Omega$.

Therefore, $g_0(s)$
is a relative $m$-center of
$\{f_1(s), \ldots,f_n(s)\}\subset X$ in $Y$
for almost all $s\in\Omega$.

\end{proof}

\section{Main results}

The notion of relative $p$-center in $Y$ for almost all $s\in\Omega$, implies the
relative $p$-center in $L_p(\mu, Y)$, when
$(\Omega, \Sigma, \mu)$ is a complete probability space and $1\leq p<+\infty$.

\begin{prop}\label{pro:2}
Let $X$ be a Banach space and $Y$ a closed subspace of $X$. Given any functions
$f_1,\ldots,f_n\in L_p(\mu, X)$ and $g\in L_p(\mu, Y)$.
 Then if
$g(s)$ is a relative $p$-center of
$\{f_1(s), \ldots,f_n(s)\}\subset X$ in $Y$,
for almost all $s\in\Omega$,
$g$ is a relative $p$-center of
$\{f_1,\ldots, f_n\}\subset L_p(\mu, X)$ in $L_p(\mu, Y)$.
\end{prop}

\begin{proof}
Let $h\in L_p(\mu,Y)$. Since $g(s)$ is a relative $p$-center of
$\{f_1(s),\ldots, f_n(s)\}$$\subset X$ in $Y$, for almost all $s\in\Omega$, one has
\begin{align*}
\sum_{i=1}^n\|f_i(s)-g(s)\|^p\leq \sum_{i=1}^n\|f_i(s)-h(s)\|^p,
\end{align*}
for almost all $s\in\Omega$.
Thus
\begin{align*}
\sum_{i=1}^n\int_\Omega\|f_i(s)-g(s)\|^p \,d\mu(s)
\leq \sum_{i=1}^n\int_\Omega\|f_i(s)-h(s)\|^p\,d\mu(s).
\end{align*}
 Therefore, we get
\begin{align*}
\sum_{i=1}^n\|f_i-g\|_p^p\leq \sum_{i=1}^n\|f_i-h\|_p^p,
\ \ \ \forall\, h\in L_p(\mu,Y),
\end{align*}
which means that $g$ is is a relative $p$-center of
$\{f_1,\ldots,f_n\}\subset L_p(\mu, X)$ in $L_p(\mu, Y)$.
\end{proof}
Observe that  the previous proposition fails if we take relative
$m$-center with $m\neq p$, as it is shown in the following example.

\begin{exa}\label{exa:1}
By taking $\Omega=\{s_1,s_2\}$, we define $\Sigma=\mathcal{P}(\Omega)$ and
$\mu(\{s_1\})=\mu(\{s_2\})=1$. Let $X=\R^2$ with the euclidean norm and
$Y=\{(x,0): x\in\R\}$. Consider the functions
 $f_1,f_2,g:\Omega\longrightarrow X$
    \begin{align*}
        & f_1(s_1):=(-1,1), \ \ \ f_1(s_2):=(2,2)\\
        & f_2(s_1):=(2,2), \ \ \ f_2(s_2):=(-1,1)\\
        & g(s_1):=\Big(\frac{1}{2},0\Big), \ \ \ g(s_2):=\Big(\frac{1}{2},0\Big).
    \end{align*}
Then $g(s)$ is a relative $2$-center of
$\{f_1(s), f_2(s)\}\subset X$ in $Y$,
for all $s\in\Omega$, but $g$ is not a relative $2$-center of
$\{f_1, f_2\}\subset L_1(\mu,X)$ in $L_1(\mu,Y)$.
\end{exa}

\begin{proof}
We have that
$$
\|f_1(s)-g(s)\|^2+\|f_2(s)-g(s)\|^2\leq \|f_1(s)-y\|^2+\|f_2(s)-y\|^2
$$
for all $y\in Y$ and all $s\in\Omega$. Hence  $g(s)$ is a relative $2$-center of
$\{f_1(s), f_2(s)\}\subset X$ in $Y$, for all $s\in\Omega$.

On the other hand
\begin{align*}
  \|f_1-&g\|_1^2+\|f_2-g\|_1^2=\bigg(\int_\Omega\|f_1(s)-g(s)\|d\mu(s)\bigg)^2+\\
 &+\bigg(\int_\Omega\|f_2(s)-g(s)\|d\mu(s)\bigg)^2=\\
 & =\big(\|f_1(s_1)-g(s_1)\|+\|f_1(s_2)+g(s_2)\|\big)^2\\
 & + \big(\|f_2(s_1)-g(s_1)\|+\|f_2(s_2)+g(s_2)\|\big)^2= \\
  &=\Bigg(\sqrt{\Big(\frac{3}{2}\Big)^2+1}+\sqrt{\Big(\frac{3}{2}\Big)^2+2^2}\Bigg)^2
  +\Bigg(\sqrt{\Big(\frac{3}{2}\Big)^2+2^2}+\sqrt{\Big(\frac{3}{2}\Big)^2+1}\Bigg)^2= \\
&=2\Bigg(\sqrt{\frac{13}{4}}+\sqrt{\frac{25}{4}}\Bigg)^2=
\frac{1}{2}(13+10\sqrt{13}+25)=19+5\sqrt{13}.
\end{align*}
We define  $h:\Omega\longrightarrow Y$ by
$$
h(s_1)=h(s_2):=(0,0).
$$
Then, \begin{align*}
 \|f_1-h\|_1^2+\|f_2-h\|_1^2&=
 \big(\|f_1(s_1)-h(s_1)\|+\|f_1(s_2)-h(s_2)\|\big)^2+\\
 &+ \big(\|f_2(s_1)-h(s_1)\|+\|f_2(s_2)-h(s_2)\|\big)^2=  \\
 &=\Big(\sqrt{2}+\sqrt{8}\Big)^2+\Big(\sqrt{8}+\sqrt{2}\Big)^2=
 36
\end{align*}
Since $5\sqrt{13}>18$ we have $36<19+5\sqrt{13}$, one has
$$
\|f_1-h\|_1^2+\|f_2-h\|_1^2<\|f_1-g\|_1^2+\|f_2-g\|_1^2,
$$
which means $g$ is not a relative $2$-center of
$\{f_1, f_2\}\subset L_1(\mu,X)$ in $L_1(\mu,Y)$.
\end{proof}

\begin{lem}\label{lemm:3}
Let $Y$ be a closed subspace of Banach space $X$ and
$f_1,\ldots,f_n\in L_p(\mu, X)$. Let
 $g:\Omega\longrightarrow Y$ be a $\mu$-measurable function such
that $g(s)$ is a relative $p$-center of
$\{f_1(s), \ldots,f_n(s)\}\subset X$ in $Y$
for almost all $s\in\Omega$. Then
$g$ is a relative $p$-center of
$\{f_1,\ldots, f_n\}\subset L_p(\mu, X)$ in $L_p(\mu, Y)$.
\end{lem}

\begin{proof}
Since $g(s)$ is a relative $p$-center of
$\{f_1(s), \ldots,f_n(s)\}\subset X$ in $Y$,
for almost all $s\in\Omega$, we have
$$
\sum_{i=1}^n\|f_i(s)-g(s)\|^p\leq \sum_{i=1}^n\|f_i(s)-y\|^p
$$
for almost all $s\in\Omega$ and each $y\in Y$.

On the other hand, for $i_0\in\{1,\ldots,n\}$ it follows
\begin{align*}
  \|g(s)\|^p&\leq\big(\|f_{i_0}(s)-g(s)\|+\|f_{i_0}(s)\|\big)^p\leq\\
           &\leq 2^{p-1}\big(\|f_{i_0}(s)-g(s)\|^p+\|f_{i_0}(s)\|^p\big)\leq\\
          &\leq 2^{p-1}\sum_{i=1}^n\|f_i(s)-g(s)\|^p+2^{p-1}\|f_{i_0}(s)\|^p\leq\\
          &\leq 2^{p-1}\sum_{i=1}^n\|f_i(s)-y\|^p+2^{p-1}\|f_{i_0}(s)\|^p,
\end{align*}
for almost all $s\in\Omega$ and each $y\in Y$.
Now by taking $y=0$, we get
$$
\|g(s)\|^p\leq  2^{p-1}\sum_{i=1}^n\|f_i(s)\|^p+2^{p-1}\|f_{i_0}(s)\|^p
$$
for almost all $s\in\Omega$.
 Thus
 $$
 \|g\|_p^p\leq 2^{p-1}\sum_{i=1}^n\|f_i\|_p^p+2^{p-1}\|f_{i_0}\|_p^p.
 $$
 Hence $g\in L_p(\mu,Y)$ and
 from Proposition \ref{pro:2}, it
 is a relative $p$-center of
$$\{f_1,\ldots, f_n\}\subset L_p(\mu, X)$$ in $L_p(\mu, Y)$.
\end{proof}

\begin{thm}\label{thm:1}
Let $X$ be a Banach space.
If \, $Y$ is a subspace of $X$, $p$-simultaneously proximinal in $X$ and weakly
$\mathcal{K}$-analytic, then
$L_p(\mu,Y)$ is $p$-simultaneously proximinal in $L_p(\mu,X)$.
\end{thm}

\begin{proof}
Let $f_1,\ldots,f_n$ be functions in $L_p(\mu,X)$. Then
Proposition
\ref{pro:1} guarantees the existence
 a $\mu$-measurable function $g_0:\Omega\longrightarrow Y$
such that $g_0(s)$ is a relative $p$-center of
$\{f_1(s), \ldots,f_n(s)\}\subset X$ in $Y$
for almost all $s\in\Omega$.
By Lemma \ref{lemm:3} it follows that
$g_0$ is a relative $p$-center of
$\{f_1, \ldots,f_n\}\subset L_p(\mu,X)$ in $L_p(\mu,Y)$.
\end{proof}

\begin{prop}
Let $(H,\|\cdot\|)$ be a Hilbert space, $1\leq p<+\infty$, $p\neq2$ and $m=2$.
Then there exists a three-point set
 $\{f_1,f_2,f_3\}\subset L_p(\mu,H)$ and $g\in L_p(\mu,H)$ such
 that
$g(s)\in H$ is a $2$-center of $\{f_1(s), f_2(s),f_3(s)\}\subset H$,
for every $s\in\Omega$. But $g$ is not a $2$-center of
$\{f_1, f_2,f_3\}$.
\end{prop}

\begin{proof}
Since the space $L_p(\mu,H)$ $(p\neq2)$ is not an inner product
space, by (Ben\'{\i}tez, Fern\'{a}ndez and Soriano, 2002), there
exists a three-point set $\{f_1,f_2,f_3\}\subset L_p(\mu,H)$ such
that
\begin{align}\label{ali:1}
Z^2(\{f_1,f_2,f_3\})\cap\co(\{f_1,f_2,f_3\})=\emptyset.
\end{align}
On the other hand, we have that the set $Z^2(\{f_1,f_2,f_3\})$ is not empty and
\begin{align}\label{ali:2}
\frac{1}{3}\big(f_1(s)+f_2(s)+f_3(s)\big)\in Z^2(\{f_1(s),f_2(s),f_3(s)\}),
\ \ \ \text{ for all } \ \  s\in \Omega.
\end{align}
We define the function $g:\Omega\longrightarrow H$ by
$$g(s):=\frac{1}{3}\big(f_1(s)+f_2(s)+f_3(s)\big).$$
By construction of $g(s)$ one has $g\in \co(\{f_1,f_2,f_3\})$. On the other hand, because of $\{f_1,f_2,f_3\}\subset L_p(\mu,H)$, we have $g\in L_p(\mu,H)$. From (\ref{ali:1}),
have to $g\notin Z^2(\{f_1,f_2,f_3\})$. However from (\ref{ali:2}),
$g(s)$ is a $2$-center of
$\{f_1(s),f_2(s) ,f_3(s)\}$
for all $s\in\Omega$. This concludes the proof.
\end{proof}

\end{document}